\documentclass[reqno]{amsart}
\usepackage{stmaryrd}
\usepackage{silence}
\usepackage{upgreek}
\usepackage{amssymb}
\usepackage{faktor}
\usepackage[all,cmtip]{xy}
\usepackage{amsmath} 
\usepackage{mathrsfs}
\usepackage{tikz}
\usepackage{tikz-cd}
\usepackage{enumitem}
\usepackage{mathtools}
\usepackage[mathcal]{euscript}
\usepackage{graphicx}
\usepackage{url}
\usepackage{hyperref}
\usepackage[T1]{fontenc}

\usetikzlibrary{shapes.geometric}
\usetikzlibrary{decorations.markings}
\usetikzlibrary{patterns,decorations.pathreplacing}

\hypersetup{colorlinks=true,linkcolor=blue}

\usepackage{ragged2e}

\newcommand{\intrinsic}[1]{S(#1)}

\newcommand{\fullyreduced}[1]{\widetilde{S}\left(#1\right)}

\DeclareMathOperator*{\im}{\mathrm{im}}

\definecolor{coloryellow}{RGB}{240,228,66}
\definecolor{colorskyblue}{RGB}{86,180,233}
\definecolor{colorvermillion}{RGB}{213,94,0}

\DeclareMathOperator{\coker}{coker}

\newcommand{\graphfont}{\mathsf}

\newcommand{\thetagraph}[1]{\graphfont{\Theta}_{#1}}

\newcommand{\stargraph}[1]{\graphfont{S}_{#1}}
\newcommand{\graf}{\graphfont{\Gamma}}

\newcommand{\lollipopgraph}[1]{\graphfont{L}_{#1}}

\newcommand{\cC}{\mathcal{C}}

\DeclareSymbolFont{sfletters}{OT1}{cmss}{m}{n}
\DeclareMathSymbol{\sTheta}{\mathord}{sfletters}{"02}

\theoremstyle{definition}
\newtheorem{definition}{Definition}[section]

\newtheorem{example}[definition]{Example}

\newtheorem{construction}[definition]{Construction}
\newtheorem{observation}[definition]{Observation}

\theoremstyle{plain}
\newtheorem{proposition}[definition]{Proposition}
\newtheorem{lemma}[definition]{Lemma}
\newtheorem{corollary}[definition]{Corollary}
\newtheorem{theorem}[definition]{Theorem}

\theoremstyle{remark}
\newtheorem{remark}[definition]{Remark}
\newtheorem{question}[definition]{Question}
\newtheorem{problem}[definition]{Problem}

\makeatletter
\@addtoreset{definition}{section}
\makeatother

\usepackage{marginnote}
    \DeclareFontFamily{U}{wncy}{}
    \DeclareFontShape{U}{wncy}{m}{n}{<->wncyr10}{}
    \DeclareSymbolFont{mcy}{U}{wncy}{m}{n}
    \DeclareMathSymbol{\Sha}{\mathord}{mcy}{"58}

\newsavebox{\foobox}

\title{Asymptotic homology of graph braid groups}
\author{Byung Hee An}
\email{anbyhee@knu.ac.kr}
\address{Department of Mathematics Education, Teachers College, Kyungpook National University, Daegu, South Korea}
\author{Gabriel C. Drummond-Cole}
\author{Ben Knudsen}
\begin{document}
\begin{abstract}
We give explicit formulas for the asymptotic Betti numbers of the unordered configuration spaces of an arbitrary finite graph over an arbitrary field.
\end{abstract}

\maketitle

\section{Introduction}

The (co)homology of configuration spaces is a classical topic of perennial interest \cite{Arnold:CRGDB, BodigheimerCohenTaylor:OHCS,Totaro:CSAV,FelixThomas:RBNCS,Church:HSCSM,DrummondColeKnudsen:BNCSS}. This article is concerned with the homology of configuration spaces of graphs \cite{FarleySabalka:DMTGBG,KoPark:CGBG,MaciazekSawicki:HGP1CG,ChettihLuetgehetmann:HCSTL,Ramos:SPHTBG,AnDrummondColeKnudsen:SSGBG}, which for our purposes are simply finite $1$-dimensional cell complexes. For a graph $\graf$, we write \[B_k(\graf)=\{(x_1,\ldots, x_k)\in\graf^k:x_i\neq x_j\text{ if } i\neq j\}/\Sigma_k\] for the $k$th unordered configuration space of $\graf$. These spaces classify their fundamental groups, the graph braid groups \cite{Abrams:CSBGG}, so this homology is also group homology. 

Experience has shown that Betti numbers of configurations spaces are profitably viewed as functions of $k$. Fixing a homological degree, our main result asserts that the resulting function is asymptotic to a simple, explicit formula in the combinatorics of $\graf$ (in the sense that the ratio of the two functions tends to $1$ as $k\to\infty$).

Recall that the degree or valence $d(w)$ of a vertex $w$ is the number of components of its complement in a small neighborhood, and $w$ is called essential if $d(w)\geq3$. Given a set of essential vertices $W$, we write $\Delta^W_\graf$ for the number of components of the complement of $W$ in $\graf$. The $i$th \emph{Ramos number} of $\graf$ is defined as the maximum $\Delta^i_\graf=\max_{|W|=i}\Delta^W_\graf$ of these numbers \cite{Ramos:SPHTBG}.

\begin{theorem}\label{thm:growth}
Fix a field $\mathbb{F}$ and $i\geq0$. If $\graf$ is a connected graph with an essential vertex and $\Delta^i_\graf>1$, then 
\[\dim H_i(B_k(\graf);\mathbb{F})\sim \sum_{W}\frac{1}{(\Delta^i_\graf-1)!}\prod_{w\in W} (d(w)-2)k^{\Delta^i_\graf-1},\] where $W$ ranges over sets of essential vertices of cardinality $i$ such that $\Delta^W_\graf=\Delta^i_\graf$.
\end{theorem}

As the authors have previously shown, the $i$th Betti number of $B_k(\graf)$ is eventually a polynomial of degree $\Delta^i_\graf-1$ \cite{AnDrummondColeKnudsen:ESHGBG}, so the theorem amounts to the calculation of the leading coefficient of this polyomial. 

We comment briefly on our hypotheses. The inequality $\Delta^i_\graf> 1$ always holds except sometimes in the case $i=1$, where it is possible that $\Delta^1_\graf=1$. This case is completely understood by work of Ko--Park \cite{KoPark:CGBG}---in particular, the conclusion of the theorem is known not to hold when $\Delta^i_\graf=1$. The case of a disconnected graph follows in light of the K\"{u}nneth formula, and a connected graph without an essential vertex is homeomorphic to a point, an interval or a circle, whose configuration spaces are easily understood.

We now contextualize Theorem \ref{thm:growth} within three lines of research of independent interest.

\subsection{Homological stability} The idea of studying the homology of configuration spaces asymptotically is an old one, with roots in the scanning and group completion techniques of McDuff \cite{McDuff:CSPNP} and Segal \cite{Segal:CSILS} and present already in the very earliest computations \cite{Arnold:CRGDB}. In more recent years, this idea has flowered into the study of homological stability and stability phenomena in general \cite{ChurchEllenbergFarb:FIMSRSG}. The analogue of classical homological stability \cite{McDuff:CSPNP,Church:HSCSM} in the context of configuration spaces of graphs is the aforementioned fact that the Betti numbers are eventually equal to polynomials in the number of particles. 

Given homological stability, several questions arise. Theorem \ref{thm:growth} is an answer to the analogue of one such question: what is the stable homology? We pose two more, the first also purely computational.

\begin{question}
What is the stable range? That is, when does the $i$th Betti number of $B_k(\graf)$ begin to equal a polynomial in $k$?
\end{question}

The second question is more conceptual, and we can only formulate it vaguely.

\begin{question}
What does the stable homology represent? Does an analogue of scanning or group completion apply?
\end{question}

\subsection{Universal presentations}\label{section:universal presentations} The first interesting example in the homology of configuration spaces of graphs is the \emph{star class}, in which two particles orbit one another by passing successively through an essential vertex (see Section \ref{section:loops stars relations}). Performing the same local move simultaneously at different essential vertices, perhaps with the addition of stationary particles on edges, produces a large family of toric classes in homology. 

It was once thought that, at least rationally, all homology might be generated by such tori, together with the homology of the graph itself. This idea was put to rest by the discovery that the configuration space of $3$ unordered particles in the graph obtained by suspending $4$ points is a compact orientable surface of genus $3$, up to homotopy \cite{WiltshireGordon:MCSSC,ChettihLuetgehetmann:HCSTL}. Other exotic homology classes have since been discovered, and we refer to the resulting zoology problem as the problem of \emph{universal generation.}

\begin{problem}\label{problem:universal generation}
Give a finite list of atomic graphs providing generators for the homology of the configuration spaces of all graphs (perhaps within a certain class) in fixed degree.
\end{problem}

For example, star classes together with cycles in the graph itself generate in degree $1$ \cite{KoPark:CGBG}; toric classes generate in all degrees for trees and wheel graphs \cite{MaciazekSawicki:NAQSG} (see \cite{ChettihLuetgehetmann:HCSTL} for a related result in the ordered context); and, as the authors show in forthcoming work, the only ``exotic'' generator in degree $2$ is the aforementioned fundamental class of the surface of genus $3$, at least for planar graphs \cite{AnDrummondColeKnudsen:OSHPGBG}. Abstractly, it is known that Problem \ref{problem:universal generation} has a solution if one allows the formation of minors as well as subgraphs \cite{MiyataProudfootRamos:CGMT}.

In the course of proving Theorem \ref{thm:growth}, we show that the dimension of $H_i(B_k(\graf);\mathbb{F})$ is asymptotic to the dimension of the submodule spanned by toric classes satisfying a certain rigidity condition (see Section \ref{section:tori}); indeed, it is the combinatorics of this submodule that account for the asymptotic formula. Thus, although tori are not universal generators, \emph{they are so asymptotically}.

The result also touches on the companion problem of universal \emph{relations}, where it asserts that the only non-obvious relation among these asymptotic generators is the $X$-\emph{relation} of Lemma \ref{lem:stable X}.

\subsection{Torsion}\label{section:torsion} Elements of finite order in the homology of configuration spaces of graphs are exceedingly rare. At the level of ordered configuration spaces, no examples are known, and it has been proven that none exist for certain limited classes of graphs \cite{ChettihLuetgehetmann:HCSTL}. In the unordered case of interest to us, Kim--Ko--Park have shown that $2$-torsion in the first homology of the configuration space of two points detects planarity \cite{KimKoPark:GBGRAAG}. In the non-planar case, the size of the $2$-torsion subgroup is computable in terms of invariants of the background graph \cite{KoPark:CGBG}, and this $2$-torsion propagates to higher degrees and larger configurations via disjoint union and the addition of stationary particles to edges. Apart from this single source of torsion, essentially nothing is known.

\begin{question}
Is there nontrivial odd torsion in the homology of graph braid groups? Are there elements of even order greater than $2$?
\end{question}

To see the connection between this question and Theorem \ref{thm:growth}, recall that, by the universal coefficient theorem, a reasonably finite space has torsion-free homology if and only if its Betti numbers are independent of the coefficient field. For this reason, Theorem \ref{thm:growth} should be interpreted as implying that \emph{the homology of graph braid groups is asymptotically torsion-free}. 

One way of making this last statement precise is the following result, which follows immediately from Theorem \ref{thm:growth}, the universal coefficient theorem, and induction on $i$ using the observation that $\Delta^i_\graf\geq\Delta^{i-1}_\graf+1$.

\begin{corollary}
Let $\graf$ be a connected graph with an essential vertex and $\Delta^i_\graf>1$. For any prime $p$ and $k$ sufficiently large, the order of the $p$-torsion subgroup of $H_i(B_k(\graf);\mathbb{Z})$ is equal to $p^{f(k)}$, where $f(k)\in\mathbb{Q}[k]$ has degree strictly less than $\Delta^i_\graf-1$.
\end{corollary}

As indicated above, the authors are aware of no other systematic, quantitative result on odd torsion in the homology of graph braid groups.

\subsection{Idea of proof} The proof of Theorem \ref{thm:growth} proceeds in two steps. First, we show that the toric classes mentioned above account for most of the homology of the configuration spaces of a graph (Theorem \ref{thm:tame}). Second, and easier, we count tori (Theorem \ref{thm:counting tori}).

In pursuing the first task, an argument by induction on the first Betti number of $\graf$ suggests itself, since tori account for everything in the case of a tree. This inductive argument is facilitated by \emph{vertex explosion}, a technique whereby the configuration spaces of $\graf$ are related to those of a simpler graph by a long exact sequence (see Section \ref{section:vertex explosion}). In order to achieve the inductive step, we require more than knowledge of the rates of growth of the modules in question; we must be assured that these rates of growth obtain \emph{for specific geometric reasons}. This fine control is expressed in the concept of \emph{tameness}, our principal technical innovation (see Section \ref{section:tameness}).

\subsection{Relation to previous work} A stability phenomenon in $k$ for $H_i(B_k(\graf))$ was first observed by Ko--Park in the case $i=1$ \cite{KoPark:CGBG}. Ramos subsequently observed a generalization of this phenomenon for all $i$ in the case where $\graf$ is a tree, and he further identified the invariant $\Delta_\graf^i$ controlling the degree of polynomial growth in this case \cite{Ramos:SPHTBG}. The authors interpreted this stability phenomenon geometrically in terms of \emph{edge stabilization} (see Section \ref{section:swiatkowski}), permitting its generalization to all graphs, and computed the degree of growth in general \cite{AnDrummondColeKnudsen:ESHGBG}. The first author formulated Theorem \ref{thm:growth} as a conjecture in the summer of 2019 and presented it at the AIM workshop ``Configuration spaces of graphs'' in early 2020. The extensive computer calculations of the second author provided evidence for the conjecture \cite{DrummondCole:BNUCSSG}.

\subsection{Acknowledgments} 
The first author was supported by the National Research Foundation of Korea(NRF) grant funded by the Korea government(MSIT) (No. 2020R1A2C1A01003201).
The third author was supported by NSF grant DMS 1906174. The authors benefited from the hospitality of the American Institute of Mathematics during the workshop ``Configuration spaces of graphs.''

\subsection{Conventions}\label{section:conventions} Given functions $f,g:\mathbb{Z}_{\geq0}\to\mathbb{Z}_{\geq0}$, we say that $f$ and $g$ are asymptotic, written $f\sim g$, provided both are eventually nonzero and \[
\lim_{k\to\infty}\frac{f(k)}{g(k)}=1.
\]

We work over a fixed ground ring $R$, which we take to be a Noetherian integral domain. At times, we require the ground ring to be a field, writing $\mathbb{F}$ instead. Homology is taken implicitly with coefficients in the ground ring. The symbol $R\langle X\rangle$ denotes the free $R$-module on the set $X$.

All gradings are non-negative. Given a degreewise finite dimensional graded $\mathbb{F}$-vector space $V$, we write $\dim V$ for the function $k\mapsto \dim V_k$, regarded as a function from $\mathbb{Z}_{\geq0}$ to itself.

\section{Preliminaries}

In this section, we collect the facts and tools used in the proof of Theorem \ref{thm:growth}. Although largely a review, it contains a few crucial observations not appearing elsewhere, namely Corollary \ref{cor:combined X}, Lemma \ref{lem:good tori free module}, and Corollary \ref{cor:inject into associated graded}.

\subsection{The \'{S}wi\k{a}tkowski complex}\label{section:swiatkowski} This section is a brief summary of some necessary terminology and results from \cite{AnDrummondColeKnudsen:ESHGBG}.

A graph is a finite $1$-dimensional CW complex, whose $0$-cells and open $1$-cells are called vertices and edges, respectively. A contractible graph is called a tree. A half-edge is a point in the preimage of a vertex under the attaching map of a $1$-cell;
thus, every edge determines two half-edges. In general, sets of vertices, edges, and half-edges are denoted $V(\graf)$, $E(\graf)$, and $H(\graf)$, respectively, but we omit $\graf$ from the notation wherever doing so causes no ambiguity.

A half-edge $h$ has an associated vertex $v(h)$ and an associated edge $e(h)$, and we write $H(v)=\{h\in H: v=v(h)\}$ for the set of half-edges incident on $v\in V$. The degree or valence of $v$ is $d(v)=|H(v)|$. A vertex is essential if its valence is at least $3$, and an edge is a tail if its closure contains a vertex of valence $1$.

A subgraph is a subcomplex of a graph. A graph morphism is a finite composition of isomorphisms onto subcomplexes and inverse subdivisions---see \cite[\S2.1]{AnDrummondColeKnudsen:ESHGBG} for details.

\begin{definition}
Let $\graf$ be a graph. For $v\in V$, write $S(v)=\mathbb{Z}\langle\varnothing, v, h\in H(v)\rangle$. The \emph{\'{S}wi\k{a}tkowski complex} (with coefficients in $R$) is the $R[E]$-module \[\intrinsic{\graf}=R[E]\otimes_\mathbb{Z}\bigotimes_{v\in V} S(v),\] equipped with the bigrading $|\varnothing|=(0,0)$, $|v|=|e|=(0,1)$, and $|h|=(1,1)$, together with the differential determined by the equation $\partial(h)=e(h)-v(h)$.
\end{definition}

When denoting elements of $\intrinsic{\graf}$, we systematically omit all factors of $\varnothing$ and all tensor symbols, and we regard half-edge generators at different vertices as permutable up to sign. See \cite[\S2.2]{AnDrummondColeKnudsen:ESHGBG} for further discussion of $\intrinsic{\graf}$ and its elements.

Write $B(\graf)=\bigsqcup_{k\geq0}B_k(\graf)$. The following result is \cite[Thm. 2.10]{AnDrummondColeKnudsen:ESHGBG}, but see \cite{Swiatkowski:EHDCSG} and \cite{ChettihLuetgehetmann:HCSTL} for precursors.

\begin{theorem}
There is a natural isomorphism of bigraded $R[E]$-modules \[H_*(B(\graf))\cong H_*(\intrinsic{\graf}).\]
\end{theorem}

Two comments are in order. First, the action of $R[E]$ on the lefthand side arises from an $E$-indexed family of \emph{edge stabilization maps}. Stabilization at $e$ replaces the subconfiguration of particles lying in the closure of $e$ with the collection of averages of consecutive particles and endpoints---see Figure \ref{fig:edge stabilization} and \cite[\S2.2]{AnDrummondColeKnudsen:ESHGBG}. Second, regarding the implied functoriality, we direct the reader to \cite[\S2.3]{AnDrummondColeKnudsen:ESHGBG}.

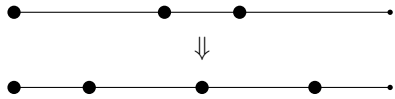
\begin{figure}[ht]
\begin{tikzpicture}
\fill[black] (0,0) circle (2.5pt);
\fill[black] (5,0) circle (1pt);
\fill[black] (2,0) circle (2.5pt);
\fill[black] (3,0) circle (2.5pt);
\draw(0,0) -- (5,0);
\draw(2.5,-.2) node[below]{$\Downarrow$};
\begin{scope}[yshift=-1cm]
\fill[black] (0,0) circle (2.5pt);
\fill[black] (5,0) circle (1pt);
\fill[black] (1,0) circle (2.5pt);
\fill[black] (2.5,0) circle (2.5pt);
\fill[black] (4,0) circle (2.5pt);
\draw(0,0) -- (5,0);
\end{scope}
\end{tikzpicture}
\caption{Edge stabilization}\label{fig:edge stabilization}
\end{figure}

The \emph{reduced} \'{S}wi\k{a}tkowski complex is obtained by replacing $S(v)$ in the definition of $\intrinsic{\graf}$ with the submodule $\fullyreduced{v}\subseteq S(v)$ spanned by $\varnothing$ and the differences of half-edges. The inclusion $\fullyreduced{\graf}\subseteq \intrinsic{\graf}$ is a quasi-isomorphism as long as $\graf$ has no isolated vertices \cite[Prop. 4.9]{AnDrummondColeKnudsen:SSGBG}. Note that, for any $h_1\in H(v)$, a basis for $\fullyreduced{v}$ is given by $\{\varnothing\}\cup \{h-h_1\}_{h\in H(v)\setminus \{h_1\}}$. In this way, a (non-canonical) basis for $\fullyreduced{\graf}$ may be obtained.

\subsection{Exploding vertices}\label{section:vertex explosion}

Given a graph $\graf$ and $v\in V$, we write $\graf_v$ for the graph obtained by exploding the vertex $v$---see Figure \ref{fig:vertex explosion} and \cite[Def. 2.12]{AnDrummondColeKnudsen:ESHGBG}---which we regard as a subgraph of a subdivision of $\graf$, uniquely up to isotopy. In general, there is a long exact sequence relating the homology of $B(\graf)$ with that of $B(\graf_v)$ \cite[Prop. 2.3]{AnDrummondColeKnudsen:ESHGBG}. We state only the special case we require.\\

\begin{figure}[ht]
\begin{tikzpicture}
\begin{scope}
\fill[black] (0,0) circle (2.5pt);
\draw(-1,0) -- (1,0);
\draw(0,-.2) node[below]{$\graf$};
\end{scope}
\begin{scope}[xshift=4cm]
\fill[black] (-.3,0) circle (2.5pt);
\fill[black] (.3,0) circle (2.5pt);
\draw(-1,0) -- (-.4,0);
\draw(1,0) -- (.4,0);
\draw(0,-.2) node[below]{$\graf_v$};
\end{scope}
\end{tikzpicture}
\caption{A local picture of vertex explosion}
\label{fig:vertex explosion}
\end{figure}
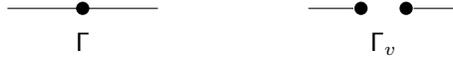

\begin{proposition}\label{prop:vertex explosion}
If $v\in V$ is a bivalent vertex with corresponding edges $e,e'\in E$ and half-edges $h,h'\in H(v)$, then the sequence \[\cdots\to H_i(B_k(\graf_v))\xrightarrow{\iota} H_i(B_k(\graf))\xrightarrow{\psi} H_{i-1}(B_{k-1}(\graf_v))\xrightarrow{e-e'} H_{i-1}(B_k(\graf_v))\to\cdots\] is exact. Here, 
\begin{enumerate}
\item the map $\iota$ is induced by the inclusion of $\graf_v$,
\item the map $\psi$ is induced by the chain map on reduced \'{S}wi\k{a}tkowski complexes sending $\beta+(h-h')\alpha$ to $\alpha$, where $\beta$ involves no half-edge generators at $v$, and
\item the map $e-e'$ is multiplication by the ring element $e-e'\in R[E]$.
\end{enumerate} Moreover, all maps shown are compatible with edge stabilization.
\end{proposition}

More generally, given a subset $W\subseteq V$, we write $\graf_W$ for the graph obtained by exploding each of the vertices in $W$. The analogue of the exact sequence of Proposition \ref{prop:vertex explosion} is a spectral sequence interpolating between the homology of $B(\graf_W)$ and that of $B(\graf)$, which arises by filtering the reduced \'{S}wi\k{a}tkowski complex by the number of vertices of $W$ needed to write an element \cite[\S3.2]{AnDrummondColeKnudsen:ESHGBG}. 

In general, this spectral sequence is rather mysterious. Fortunately, we need but little knowledge of it. Denoting the $r$th page by $E^r_{*,*}=E^r_{*,*}(W)$, we have the following result (see the proof of \cite[Lem. 3.15]{AnDrummondColeKnudsen:ESHGBG}).

\begin{proposition}\label{prop:spectral sequence}
There is a canonical $R[E]$-linear inclusion and an isomorphism \[E^\infty_{|W|,0}\subseteq E^1_{|W|,0}\cong R[\pi_0(\graf_W)]\langle X\rangle,\] where $X$ is the set of generators of $\fullyreduced{\graf}$ of the form $\bigotimes_{w\in W}h^w$, with $h^w$ a difference of half-edges at $w$.
\end{proposition}

Thus, in degree $|W|$, the top graded piece of the filtration on homology associated to $W$ is a submodule of an $R[E]$-module that is free over a specific quotient ring tied to the combinatorics of $W$ in $\graf$.

\subsection{Loops, stars, and relations}\label{section:loops stars relations} In this section, we explore two basic sources of nontrivial elements of $H_*(B(\graf))$ and the relations between them. The reader is directed to \cite[\S5.1]{AnDrummondColeKnudsen:SSGBG} for further details and proofs.

\begin{example}\label{example:loop class}
Since $\graf=B_1(\graf)$ is a subspace of $B(\graf)$, an oriented cycle in $\graf$ determines an element of $H_*(B(\graf))$, called a \emph{loop class}. 
\end{example}

A standard chain level representative of a loop class is obtained by summing the differences of half-edges involved in the cycle in question. For example, the standard representative of the unique loop class in the graph $\lollipopgraph{}$ depicted in Figure \ref{fig:lollipop}, oriented clockwise, is $h-h'\in \fullyreduced{\lollipopgraph{}}$.

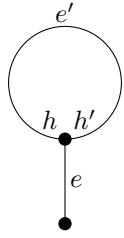
\begin{figure}[ht]
\begin{tikzpicture}[scale=.75]
\fill[black] (0,-1) circle (10/3pt);
\fill[black] (0,-2.5) circle (10/3pt);
\draw(0,0) circle (1cm);
\draw(0,-2.5) -- (0,-1);
\draw (.35,-.65) node{$h'$}; 
\draw (-.27,-.65) node{$h$}; 
\draw (0,1.25) node{$e'$}; 
\draw (0.2,-1.75) node{$e$}; 
\end{tikzpicture}
\caption{The lollipop graph $\lollipopgraph{}$}
\label{fig:lollipop}
\end{figure}

\begin{example}\label{example:star class}
There is a canonical homotopy equivalence $S_1\simeq B_2(\stargraph{3})$, where $\stargraph{3}$ is the cone on three points---see Figure \ref{fig:star class} for a depiction of the cycle witnessing this homotopy equivalence. By functoriality, then, the choice of half-edges $h_1$, $h_2$, and $h_3$ sharing a common vertex determines a \emph{star class} in $H_1(B_2(\graf))$, which depends on the ordering only up to sign.
We will denote star classes by $\alpha$ or, e.g., $\alpha_{123}$ if we wish to emphasize the particular choice of half-edges.
\end{example}

Writing $e_j$ for the edge associated to $h_j$, a standard chain level representative for a star class is given by the sum \[a=e_3(h_1-h_2)+e_2(h_3-h_1)+e_1(h_2-h_3).\] This expression is symmetric and exhibits well the geometry of the class (see Figure \ref{fig:star class}), but the alternative expression \[a=(e_1-e_3)(h_2-h_1)-(e_1-e_2)(h_3-h_1)\] in the basis for $\fullyreduced{\graf}$ privileging $h_1$ is also useful.

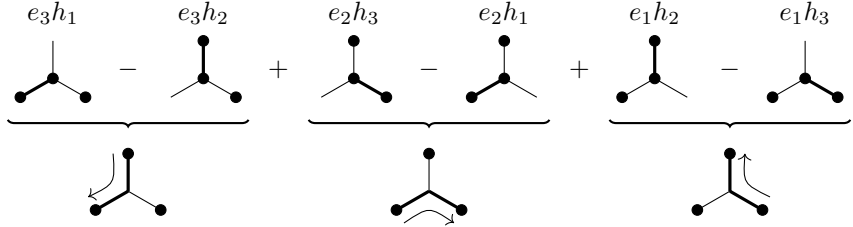
\begin{figure}[ht]
\begin{center}
\begin{tikzpicture}
\draw (1,.125) node{$-$}; 
\draw (3,.125) node{$+$};
\draw (5,.125) node{$-$};
\draw (7,.125) node{$+$};
\draw (9,.125) node{$-$};
\draw [
    thick,
    decoration={
        brace,
        mirror,
        raise=0.05cm
    },
    decorate
] (-.6,-.5) -- (2.6,-.5);
\draw [
    thick,
    decoration={
        brace,
        mirror,
        raise=0.05cm
    },
    decorate
] (3.4,-.5) -- (6.6,-.5);
\draw [
    thick,
    decoration={
        brace,
        mirror,
        raise=0.05cm
    },
    decorate
] (7.4,-.5) -- (10.6,-.5);
\begin{scope}[xshift=4cm]
\draw(0,.6) node [above] {$e_2h_3$};
\draw[black, fill=black] (0,0) circle (1.5pt);
 \draw(0,0) -- (0,.5);
 \draw(0,0) --(.433,-.25);
 \draw(0,0) --(-.433,-.25);
\draw [fill] (0,0) circle [radius = 2pt];
\draw [fill] (.433,-.25) circle [radius = 2pt];
\draw [black, very thick] (0,0)-- (.433,-.25);
\draw [fill] (0,.5) circle [radius = 2pt];
\end{scope}
\begin{scope}[xshift=0cm]
\draw(0,.6) node [above] {$e_3h_1$};
\draw[black, fill=black] (0,0) circle (1.5pt);
 \draw(0,0) -- (0,.5);
 \draw(0,0) --(.433,-.25);
 \draw(0,0) --(-.433,-.25);
\draw [fill] (0,0) circle [radius = 2pt];
\draw [fill] (.433,-.25) circle [radius = 2pt];
\draw [black, very thick] (0,0)-- (-.433,-.25);
\draw [fill] (-.433,-.25) circle [radius = 2pt];
\end{scope}
\begin{scope}[xshift=8cm]
\draw(0,.6) node [above] {$e_1h_2$};
 \draw[black, fill=black] (0,0) circle (1.5pt);
 \draw(0,0) -- (0,.5);
 \draw(0,0) --(.433,-.25);
 \draw(0,0) --(-.433,-.25);
\draw [fill] (0,0) circle [radius = 2pt];
\draw [fill] (-.433,-.25) circle [radius = 2pt];
\draw [black, very thick] (0,0)-- (0,.5);
\draw [fill] (0,.5) circle [radius = 2pt];
\end{scope}
\begin{scope}[xshift=10cm]
\draw(0,.6) node [above] {$e_1h_3$};
\draw[black, fill=black] (0,0) circle (1.5pt);
 \draw(0,0) -- (0,.5);
 \draw(0,0) --(.433,-.25);
 \draw(0,0) --(-.433,-.25);
\draw [fill] (0,0) circle [radius = 2pt];
\draw [fill] (.433,-.25) circle [radius = 2pt];
\draw [black, very thick] (0,0)-- (.433,-.25);
\draw [fill] (-.433,-.25) circle [radius = 2pt];
\end{scope}
\begin{scope}[xshift=2cm]
\draw(0,.6) node [above] {$e_3h_2$};
 \draw[black, fill=black] (0,0) circle (1.5pt);
 \draw(0,0) -- (0,.5);
 \draw(0,0) --(.433,-.25);
 \draw(0,0) --(-.433,-.25);
\draw [fill] (0,0) circle [radius = 2pt];
\draw [fill] (.433,-.25) circle [radius = 2pt];
\draw [black, very thick] (0,0)-- (0,.5);
\draw [fill] (0,.5) circle [radius = 2pt];
\end{scope}
\begin{scope}[xshift=6cm]
\draw(0,.6) node [above] {$e_2h_1$};
\draw[black, fill=black] (0,0) circle (1.5pt);
 \draw(0,0) -- (0,.5);
 \draw(0,0) --(.433,-.25);
 \draw(0,0) --(-.433,-.25);
\draw [fill] (0,0) circle [radius = 2pt];
\draw [fill] (0,.5) circle [radius = 2pt];
\draw [black, very thick] (0,0)-- (-.433,-.25);
\draw [fill] (-.433,-.25) circle [radius = 2pt];
\end{scope}
\begin{scope}[yshift=-1.5cm, xshift=9cm]
\draw(0,0) -- (0,.5);
\draw(0,0) --(.433,-.25);
\draw(0,0) --(-.433,-.25);
\draw[->] (.533,-.0768) .. controls (.1732,.1) .. (.2,.5);
\draw [fill] (0,.5) circle [radius = 2pt];
\draw [fill] (-.433,-.25) circle [radius = 2pt];
\draw [fill] (.433,-.25) circle [radius = 2pt];
\draw [black, very thick] (0,0)-- (.433,-.25);
\draw [black, very thick] (0,0)-- (0,.5);
\end{scope}
\begin{scope}[yshift=-1.5cm, xshift=5cm]
\draw(0,0) -- (0,.5);
\draw(0,0) --(.433,-.25);
\draw(0,0) --(-.433,-.25);
\draw[->](-.333,-.4232) .. controls (0,-.2) .. (.333,-.4232);
\draw [fill] (0,.5) circle [radius = 2pt];
\draw [fill] (-.433,-.25) circle [radius = 2pt];
\draw [fill] (.433,-.25) circle [radius = 2pt];
\draw [black, very thick] (0,0)-- (-.433,-.25);
\draw [black, very thick] (0,0)-- (.433,-.25);
\end{scope}
\begin{scope}[yshift=-1.5cm, xshift=1cm]
\draw(0,0) -- (0,.5);
\draw(0,0) --(.433,-.25);
\draw(0,0) --(-.433,-.25);
\draw[->](-.2,.5) .. controls (-.1732,.1) .. (-.533,-.0768);
\draw [fill] (0,.5) circle [radius = 2pt];
\draw [fill] (-.433,-.25) circle [radius = 2pt];
\draw [fill] (.433,-.25) circle [radius = 2pt];
\draw [black, very thick] (0,0)-- (-.433,-.25);
\draw [black, very thick] (0,0)-- (0,.5);
\end{scope}
\end{tikzpicture}
\end{center}
\caption{The basic star class and its standard representative}
\label{fig:star class}
\end{figure}

Star classes and loop classes interact according to a relation called the \emph{Q-relation}. Our notation will refer to the graph $\lollipopgraph{}$ of Figure \ref{fig:lollipop}, but functoriality propagates the relation to any graph with a subgraph isomorphic to a subdivision of $\lollipopgraph{}$. Writing $\gamma$ for the clockwise oriented loop class and $\alpha$ for the counterclockwise oriented star class in $\lollipopgraph{}$, we have the following.

\begin{lemma}[Q-relation]\label{lem:Q}
In the homology of the configuration spaces of the graph $\lollipopgraph{}$, there is the relation
$(e-e')\gamma=\alpha$.
\end{lemma}

The Q-relation implies that, modulo star classes, a loop class is annihilated by the appropriate difference of edges. Consequently, loop classes generate $R[E]$-submodules exhibiting strictly slower growth (again modulo star classes), hence contributing nothing asymptotically. Thus, despite its simplicity, the Q-relation already carries the germ of Theorem \ref{thm:growth}.

\begin{remark}
Roughly speaking, Theorem \ref{thm:tame} below asserts that every generator in $H_*(B(\graf))$ satisfies a relation analogous to the Q-relation forcing slow growth modulo tori, and the difficulty of the argument lies in having to establish this fact without having access to the generators or relations themselves. As discussed in Section \ref{section:universal presentations}, it would be extremely interesting to have such access.
\end{remark}

Different star classes at a vertex of valence at least $4$ are also related to one another. The first relation is essentially obvious.

\begin{lemma}\label{lem:unstable X}
In the homology of the configuration spaces of a graph containing a vertex with distinct half-edges $h_1$, $h_2$, $h_3$, and $h_4$, there is the relation of star classes
$\alpha_{123}-\alpha_{124}+\alpha_{134}-\alpha_{234}=0$.
\end{lemma}

The second relation, called the \emph{X-relation}, has almost the same form. 

\begin{lemma}[$X$-relation]\label{lem:stable X}
In the homology of the configuration spaces of a graph containing a vertex with distinct half-edges $h_1$, $h_2$, $h_3$, and $h_4$, we have the relation of stabilized star classes
$e_4\alpha_{123}-e_3\alpha_{124}+e_2\alpha_{134}-e_1\alpha_{234}=0$.
\end{lemma}

These relations combine to give the following simple but important relation. It is this relation that is responsible for the factor of $d(w)-2$ in the statement of Theorem \ref{thm:growth} (see the proof of Lemma \ref{lem:quotient by good tori}).

\begin{corollary}\label{cor:combined X}
In the homology of the configuration spaces of a graph containing a vertex with distinct half-edges $h_1$, $h_2$, $h_3$, and $h_4$, there is the relation of stabilized star classes
$(e_4-e_1)\alpha_{123}-(e_3-e_1)\alpha_{124}+(e_2-e_1)\alpha_{134}=0$.
\end{corollary}

In situations with a certain amount of degeneracy, star classes at distinct vertices can also be related. Referring to the graph $\thetagraph{3}$ pictured in Figure \ref{fig:theta}, and writing $\alpha$ and $\alpha'$ for the clockwise oriented star classes at the top vertex and bottom vertices, respectively, we have the following relation.

\begin{figure}[ht]
\begin{tikzpicture}[scale=1.5]
\fill[black] (0,-.5) circle (5/3pt);
\fill[black] (0,.5) circle (5/3pt);
\draw(0,0) circle (.5cm);
\draw(0,.5) -- (0,-.5);
\end{tikzpicture}
\caption{The theta graph $\thetagraph{3}$}
\label{fig:theta}
\end{figure}
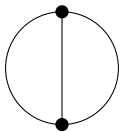

\begin{lemma}[$\thetagraph{}$-relation]\label{lem:theta}
In the homology of the configuration spaces of the graph $\thetagraph{3}$, there is the relation of star classes
$\alpha-\alpha'=0$.
\end{lemma}

This relation, which is an easy consequence of the Q-relation, gives rise to two phenomena of interest: first, the dichotomy of rigidity and non-rigidity and its implications for rates of growth (Proposition \ref{prop:tori facts}); second, $2$-torsion in the homology of non-planar graph braid groups. As discussed in Section \ref{section:torsion}, Theorem \ref{thm:growth} implies a more general causal relationship between growth and torsion, the precise nature of which remains mysterious.

\subsection{Tori}\label{section:tori} Since $B(\graf_1\sqcup \graf_2)\cong B(\graf_1)\times B(\graf_2)$, one can form the \emph{external product} of classes supported in disjoint subgraphs of a larger graph \cite[Def. 5.10]{AnDrummondColeKnudsen:SSGBG}. 

\begin{definition}
Let $W$ be a set of essential vertices. A class $\alpha\in H_{|W|}(B_{2|W|}(\graf))$ is called a \emph{$W$-torus} if it is the external product of star classes at each $w\in W$.
\end{definition}

These classes have the following fundamental property.

\begin{observation}\label{obs:torus quotient} If $e_1,e_2\in E$ can be connected by a path of edges avoiding $W$, then $(e_1-e_2)\alpha=0$ for any $W$-torus $\alpha$. 
\end{observation}

Therefore, the action of $R[E]$ on $\alpha$ factors through the quotient shown in the commuting diagram 
\[\begin{tikzcd}
R[E]\langle\alpha\rangle\ar{d}\ar{r}&R[E]\cdot\alpha\subseteq H_{|W|}(B(\graf))\\
R[\pi_0(\graf_W)]\langle \alpha\rangle\ar[dashed]{ur}&
\end{tikzcd}
\] It turns out that $W$-tori for which this map is an isomorphism can be partially characterized. 

\begin{definition}[{\cite[Def. 3.10]{AnDrummondColeKnudsen:ESHGBG}}]
The set $W$ of essential vertices is \emph{well-separating} if the open star of each $w\in W$ intersects more than one component of $\graf_W$.
\end{definition}

Recall that $E^r_{*,*}$ denotes the $r$th page of the spectral sequence arising from the filtration induced by $W$ (see Section \ref{section:vertex explosion}).

\begin{proposition}[{\cite[Lem. 3.15]{AnDrummondColeKnudsen:ESHGBG}}]\label{prop:tori facts}
Let $W$ be a well-separating set of essential vertices and $\alpha$ a $W$-torus. The following conditions are equivalent.
\begin{enumerate}
\item The natural map $R[\pi_0(\graf_W)]\langle\alpha\rangle\to R[E]\cdot\alpha$ is an isomorphism.
\item The image of $\alpha$ in $E^\infty_{|W|,0}$ is nonzero.
\item Any cycle in $\fullyreduced{\graf}$ representing a star factor of $\alpha$ involves half-edges in multiple components of $\graf_W$.
\end{enumerate}
\end{proposition}

When the equivalent conditions of Proposition \ref{prop:tori facts} hold, we say that $\alpha$ is \emph{rigid}. Concretely, the second condition means that $\alpha$ does not arise from $\graf_w$ for any $w\in W$; indeed, by definition, the $(|W|-1)$st filtration stage is precisely the span of the images of $H_*(B(\graf_w))$ for $w\in W$. 

These notions relate to the Ramos numbers as follows \cite[Cor. 3.16, Lem. 3.18]{AnDrummondColeKnudsen:ESHGBG}.

\begin{proposition}\label{prop:well-separating}
Let $W$ be a set of essential vertices of cardinality $i$.
\begin{enumerate}
\item If $W$ is well-separating, then a rigid $W$-torus exists.
\item If $\Delta^W_\graf=\Delta^i_\graf$, then either $W$ is well-separating or $i=\Delta^i_\graf=1$.
\end{enumerate}
\end{proposition}

Together, Propositions \ref{prop:tori facts} and \ref{prop:well-separating} imply the lower bound in the following result \cite[Thm. 1.2]{AnDrummondColeKnudsen:ESHGBG}, which was the authors' original motivation for considering rigid tori.

\begin{theorem}\label{thm:polynomial degree}
Under the conditions of Theorem \ref{thm:growth}, the dimension of $H_i(B_k(\graf);\mathbb{F})$ coincides with a polynomial of exact degree $\Delta^i_\graf-1$ for $k$ sufficiently large.
\end{theorem}

In particular, $\dim H_i(B_k(\graf);\mathbb{F})$ is asymptotic to a positive constant multiple of $k^{\Delta^i_\graf-1}$, so Theorem \ref{thm:growth} amounts to the identification of this constant. Building on the fact expressed in Proposition \ref{prop:tori facts} that rigid tori achieve the highest possible rate of growth, we will show that all other classes exhibit strictly slower growth, at least modulo tori (Theorem \ref{thm:tame}).

The remainder of this section is devoted to producing a convenient---in particular, independent---set of rigid tori.

\begin{construction}\label{construction:good tori}
Let $W$ be a well-separating set of essential vertices. For each $w\in W$, we fix two edges incident on $w$ and lying in distinct components of $\graf_W$. Abusively, we write $A_w$ for the set of $d(w)-2$ star classes at $w$ (with signs chosen arbitrarily) involving these two edges. We write $A_W\cong \prod_{w\in W} A_w$ for the resulting set of $W$-tori and $A_W(\graf)\subseteq H_{|W|}(B(\graf))$ for the $R[E]$-submodule generated by $A_W$. Note that each $\alpha\in A_W$ is rigid by construction.
\end{construction}

\begin{lemma}\label{lem:good tori free module}
If $W$ is a well-separating set of essential vertices, then the natural map \[R[\pi_0(\graf_W)]\langle A_W\rangle\to A_W(\graf)\] is an isomorphism.
\end{lemma}
\begin{proof}
We will show that the kernel of the composite \[R[E]\left\langle A_W\right\rangle\xrightarrow{\pi_W}R[\pi_0(\graf_W)]\left\langle A_W\right\rangle\to A_W(\graf)\subseteq H_{|W|}(B(\graf))\] is the kernel of $\pi_W$. We proceed by considering a $(|W|+1)$-chain $b$ whose boundary witnesses an element in the kernel, writing \[\partial b=\sum_I p_I(E)\bigotimes_{w\in W} a^w_{i_{w}}.\] Here the summation runs over multi-indices $I=(i_w)_{w\in W}$ with $1\leq i_w\leq d(w)-2$, $p_I(E)$ is a polynomial in the edges, and $a^w_i$ is the standard representative of $\alpha^w_i$, i.e., \[a^w_i=(e_w-e_w')h^w_i-(e_{w}-e_{w,i}')h^w,\] where $e_w$ and $e_w'$ are the fixed edges lying in distinct components of $\pi_0(\graf_W)$, $e'_{w,i}$ is the $i$th edge different from these two, and $h_w$ is the difference of half-edges from $e_w$ to $e_w'$ (resp. $h_i^w$, $e_w$, $e_{w,i}'$). Note that three of the edges and one of the half edge differences in the righthand term are independent of $i$. 

Substituting this expression, we see that, for each $I$, the coefficient of $\bigotimes_{w\in W} h^w_{i_w}$ in $\partial b$ is equal to \[p_I(E)\prod_{w\in W} (e_w-e_w').\] We will argue that this coefficient lies in $\ker(\pi_W)$, which will imply that $p_I(E)\in \ker(\pi_W)$, as desired. Indeed, the kernel is a prime ideal, as $R[\pi_0(\graf_W)]$ is an integral domain, and $e_w-e_w'\notin \ker(\pi_W)$ by construction.

We examine the bounding chain $b$. Necessarily, $b$ is of the form \[b \equiv \sum_s q_s(E)h_s\otimes\bigotimes_{w\in W}h^w_s \pmod{F_{|W|-1}},\]
where $h^w_s$ is a half-edge generator at $w$, $h_s$ is a half-edge generator at a vertex not in $W$, and $F_p=F_p\fullyreduced{\graf}$ denotes the filtration induced by $W$ (as in Section \ref{section:vertex explosion}); indeed, any term in $b$ involves exactly $|W|+1$ half-edge generators, and any term not involving every vertex of $W$ lies in $F_{|W|-1}$ by definition. Then, writing $\partial h_s=e_s-e_s'$, we obtain the equivalence \[
\partial b
\equiv
\sum_sq_s(E)(e_s-e_s')\otimes\bigotimes_{w\in W}h^w_s
\pmod{F_{|W|-1}}
\]
The filtration splits at the level of graded modules, so we actually have the equality
\[
\partial b
=
\sum_sq_s(E)(e_s-e_s')\otimes\bigotimes_{w\in W}h^w_s,\]
since each side is a sum of terms involving every vertex of $W$. Now, $e_s$ and $e'_s$ lie in a common component of $\graf_W$ for each $s$, so the quotient map $\pi_W$ annihilates the difference $e_s-e_s'$ and hence the product $q_s(E)(e_s-e_s')$. Thus, every coefficient of $\partial b$ lies in $\ker(\pi_W)$, as was to be shown.
\end{proof}

We record a consequence of this result for later use.

\begin{corollary}\label{cor:inject into associated graded}
The composite map $A_W(\graf)\subseteq  H_{|W|}(B(\graf))\to E^\infty_{|W|,0}$ is injective.
\end{corollary}
\begin{proof}
It suffices to show injectivity after further composing with the inclusion $E^\infty_{|W|,0}\subseteq E^1_{|W|,0}$ of Proposition \ref{prop:spectral sequence}. The full composite sends $\alpha\in A_W$ to the image of its standard representative in the associated graded module. The proof of Lemma \ref{lem:good tori free module} shows that the resulting classes in $E^1_{|W|,0}$ are linearly independent over $R[\pi_0(\graf_W)]$, implying the claim.
\end{proof}

\section{Tori dominate}

In pursuing Theorem \ref{thm:growth}, our most difficult task is to estimate the size of the quotient $H_i(B(\graf))/T_i(\graf)$, where $T_i(\graf)$ is the submodule generated by all $W$-tori with $|W|=i$ (see Section \ref{section:tameness modulo tori}). This quotient vanishes in the case of a tree, so we induct on the first Betti number of $\graf$. In order to carry out this induction, it will be necessary to retain a certain degree of module-theoretic control over the quotient. We begin by specifying the nature of this control.

\subsection{Tame partitions}\label{section:tameness}

Recall that a partition of a set $X$ is an unordered collection $P$ of pairwise disjoint nonempty subsets of $X$ (called blocks) whose union is $X$. We write $x\sim_P y$ when $x$ and $y$ lie in the same block of $P$. The set of partitions of $X$ forms a partially ordered set by declaring that $P\leq P'$ if and only if every block of $P'$ is contained in some block of $P$---thus, the maximal partition is the discrete partition consisting of the singletons, and the minimal partition is the trivial partition with a single block.

\begin{example}
Given a graph $\graf$ and a set $W$ of essential vertices, we write $\pi_0(\graf_W)$ for the partition of $E$ such that $e_1\sim_{\pi_0(\graf_W)} e_2$ if and only if $e_1$ and $e_2$ lie in the in the same connected component of $\graf_W$.
\end{example}

We now formulate our main definition.

\begin{definition}\label{def:tame}
Let $\graf$ be a graph. A nontrivial partition $P$ of $E$ is called $i$-\emph{tame} (with respect to $\graf$) if there is a set of essential vertices $W$ of cardinality at most $i$ with $P\leq \pi_0(\graf_W)$ such that, if $W$ is well-separating, then \begin{enumerate}
\item $P\neq  \pi_0(\graf_W)$, and
\item if $e_1$ and $e_2$ are tails and $e_1\sim _P e_2$, then $e_1\sim_{\pi_0(\graf_W)}e_2$.
\end{enumerate}
By convention, the trivial partition is $i$-tame for every $i>0$.
\end{definition}

There are no $0$-tame partitions. The convention regarding the trivial partition is necessary only when $i=\Delta^i_\graf=1$.

\begin{figure}[ht]
\begin{tikzpicture}
\fill[black] (0,0) circle (2.5pt);
\fill[black] (2,0) circle (2.5pt);
\fill[black] (-2,0) circle (2.5pt);
\fill[black] (-3,1) circle (2.5pt);
\fill[black] (-3,-1) circle (2.5pt);
\fill[black] (3,-1) circle (2.5pt);
\fill[black] (3,1) circle (2.5pt);
\fill[black] (0,1) circle (2.5pt);
\draw(-3,1) -- (-2,0);
\draw(-3,-1) -- (-2,0);
\draw(0,0) -- (-2,0);
\draw(0,0) -- (2,0);
\draw(3,1) -- (2,0);
\draw(3,-1) -- (2,0);
\draw(0,0) -- (0,1);
\draw (-3.3,-1.3) node{$A$}; 
\draw (-3.3,1.3) node{$B$}; 
\draw (-1.7,.3) node{$C$}; 
\draw (0,-.4) node{$D$}; 
\draw (0,1.4) node{$E$}; 
\draw (1.7,.3) node{$F$}; 
\draw (3.3,1.3) node{$G$}; 
\draw (3.3,-1.3) node{$H$}; 
\end{tikzpicture}
\caption{}
\label{fig:partitions}
\end{figure}

\begin{example}\label{example:partitions}
Consider the graph depicted in Figure \ref{fig:partitions}. 
\begin{enumerate}
\item The partition with blocks $\{AC\}$, $\{BC\}$, $\{CD\}$, $\{DE\}$, $\{DF, FG, FH\}$ is $\pi_0(\graf_{\{C,D\}})$. This partition is not $i$-tame for any $i$ (but see Example \ref{example:module though}).
\item\label{the example} The partition with blocks $\{AC\}$, $\{BC, CD\}$, $\{DE\}$, $\{DF, FG, FH\}$ is $2$-tame, since it is obtained from $\pi_0(\graf_{\{C,D\}})$ by merging the blocks containing $BC$ and $CD$, which are not both leaves.
\end{enumerate}
\end{example}

We record a simple fact about tame partitions, which is immediate from the definition.

\begin{lemma}\label{lem:tame move up}
If $P$ is $i$-tame, then $P$ is $(i+1)$-tame.
\end{lemma}

Given a graph $\graf$, we work in the ambient Abelian category of finitely generated $R[E]$-modules. This category receives a contravariant functor from the partially ordered set of partitions of $E$, which sends $P$ to $R[P]$ and $P\leq P'$ to the quotient map $R[P']\to R[P]$.

\begin{lemma}\label{lem:basic growth}
Let $\mathbb{F}$ be a field. For any partition $P$ of $E$, the function $\dim \mathbb{F}[P]:k\mapsto\dim(\mathbb{F}[P])_k$ is a polynomial of degree $|P|-1$.
\end{lemma}

We will be concerned with certain subcategories of the ambient Abelian category. Recall that a subcategory of an Abelian category is called a \emph{Serre subcategory} if it is nonempty, full, and closed under subobjects, quotients, and extensions.

\begin{definition}
Let $\graf$ be a graph. An $R[E]$-module $M$ is called $i$-\emph{tame} (with respect to $\graf$) if it belongs to the Serre subcategory $\cC_R(\graf,i)$ generated by graded shifts of the modules $R[P]$, where $P$ ranges over $i$-tame partitions of $E$.
\end{definition}

From our point of view, the main interest in tameness lies in controlling growth.

\begin{definition}\label{def:small}
We say that a graded vector space $M$ is $n$-\emph{small} if \[\lim_{k\to\infty}\frac{1}{k^n}\dim M_k=0.\]
\end{definition}

\begin{proposition}\label{prop:tame growth}
Let $\mathbb{F}$ be a field. If $M$ is an $i$-tame $\mathbb{F}[E]$-module and either $i>1$ or $\Delta^i_\graf>1$, then $M$ is $(\Delta^i_\graf-1)$-small.
\end{proposition}
\begin{proof}
By rank-nullity, the category of $(\Delta^i_\graf-1)$-small $\mathbb{F}[E]$-modules is itself a Serre subcategory; therefore, we may assume that $M=\mathbb{F}[P]$ for some $i$-tame partition $P$ of $E$. Since $\Delta^W_\graf\leq\Delta^i_\graf$ whenever $|W|\leq i$, it follows from Proposition \ref{prop:well-separating} and the definition of tameness that $|P|< \Delta^i_\graf$, and the claim follows from Lemma \ref{lem:basic growth}.
\end{proof}

Thus, if $M$ is $i$-tame, then the growth of $\dim M_k$ is slower than that of the submodule generated by a $W$-torus with $|W|=i$ and $\Delta^W_\graf=\Delta^i_\graf$, hence slower than that of $H_i(B(\graf))$. The advantage of tameness over smallness is its geometric nature, which permits tameness with respect to $\graf$ to be inferred from tameness with respect to a simpler graph---see Lemma \ref{lem:tame quotient}, for example.

We close with three simple criteria for tameness.

\begin{lemma}\label{lem:inequality tame}
If $P'$ is an $i$-tame partition and $P\leq P'$, then $R[P]\in \cC_R(\graf,i)$.
\end{lemma}
\begin{proof}
The module $R[P]$ is a quotient of $R[P']\in \cC_R(\graf, i)$, and $\cC_R(\graf,i)$ is closed under quotients.
\end{proof}

\begin{lemma}\label{lem:add a vertex} 
Let $W$ be a proper subset of the essential vertices of the connected graph $\graf$. Then $R[\pi_0(\graf_W)] \in \cC_R(\graf,|W|+1)$.
\end{lemma}
\begin{proof}
If $W=\varnothing$, then $\pi_0(\graf_W)$ is the trivial partition, which is $1$-tame. Otherwise, define $W'$ by adding any vertex $w$ to $W$ that is adjacent to one of its vertices. The edge connecting these two vertices is a singleton in $\pi_0(\graf_{W'})$ and we define a new partition $P$ by identifying this block with the block of any other edge at $w$. Then $P$ is $(|W|+1)$-tame, and $\pi_0(\graf_W)\leq P$, so the claim follows from Lemma \ref{lem:inequality tame}.
\end{proof}

\begin{lemma}\label{lem:non-rigid tame}
If $\alpha$ is a non-rigid $W$-torus, then $R[E]\cdot\alpha\in \cC_R(\graf, |W|)$.
\end{lemma}
\begin{proof}
By Observation \ref{obs:torus quotient}, $R[E]\cdot\alpha$ receives a surjection from $R[\pi_0(\graf_W)]$. If $W$ is not well-separating, the partition $\pi_0(\graf_W)$ is $|W|$-tame by definition, so assume otherwise. By Proposition \ref{prop:tori facts}, $\alpha$ has a star factor---at the vertex $w\in W$, say---whose half-edges all lie in a single component of $\graf_W$. The $\thetagraph{}$-relation implies that $R[E]\cdot\alpha$ admits a surjection from $R[\pi_0(\graf_{W\setminus\{w\}})]$, which is $|W|$-tame by Lemma \ref{lem:add a vertex}.
\end{proof}

\begin{example}\label{example:module though}
Although the partition $\pi_0(\graf_{\{C,D\}})$ of Example \ref{example:partitions} is not $i$-tame for any $i$, Lemma~\ref{lem:non-rigid tame} shows that the module $R[\pi_0(\graf_{\{C,D\}})]$ is $3$-tame.
\end{example}

\subsection{Tameness modulo tori}\label{section:tameness modulo tori}

Our present goal is to prove that the inclusion of the submodule generated by tori is an isomorphism after localizing at the subcategory of tame modules. Fixing $W$, we write $T_W(\graf)\subseteq H_i(B(\graf))$ for the $R[E]$-submodule generated by all $W$-tori. We write $T_i(\graf)$ for the $R[E]$-span of the $T_W(\graf)$, where $W$ ranges over all subsets of essential vertices of cardinality $i$.

\begin{theorem}\label{thm:tame}
If $\graf$ is a connected graph with an essential vertex, then the quotient $H_i(B(\graf))/T_i(\graf)$ is $i$-tame.
\end{theorem}

The strategy of the proof of Theorem \ref{thm:tame} is induction on $b_1(\graf)$. For the induction step, we will choose an edge $e$ of $\graf$, both vertices of which are essential, and subdivide it by adding a bivalent vertex $v$. We denote the resulting edges of $\graf_v$ by $e$ and $e'$, so that the natural map $E(\graf_v)\to E(\graf)$ is the quotient identifying $e$ and $e'$ in the source with $e$ in the target. Given a partition $P$ of $E(\graf_v)$, write  $P_{e\sim e'}$ for the partition of $E(\graf)$ obtained by identifying the respective blocks of $P$ containing $e$ and $e'$. 

\begin{lemma}\label{lem:tame quotient partitions}
If $P$ is $i$-tame with respect to $\graf_v$, then $P_{e\sim e'}$ is $i$-tame with respect to $\graf$. 
\end{lemma}
\begin{proof}
Let $W$ be a set of essential vertices witnessing $P$ as $i$-tame with respect to $\graf_v$. We will show that $W$ also witnesses $P_{e\sim e'}$ as $i$-tame with respect to $\graf$. Since \[P_{e\sim e'}\leq \pi_0((\graf_v)_W)_{e\sim e'}=\pi_0(\graf_W),\] we may assume that $W$ is well-separating in $\graf$, hence in $\graf_v$. There are now two cases.

If $e\sim_P e'$, then $e$ and $e'$ lie in the same component of $(\graf_v)_W$, since $e$ and $e'$ are tails of $\graf_v$. Thus, the inclusion of $(\graf_v)_W$ into $\graf_W$ induces a bijection on connected components. Given tails $e_1$ and $e_2$ of $\graf$ lying in the same block of $P_{e\sim e'}$, it follows that $e_1\sim_P e_2$, hence $e_1\sim_{\pi_0((\graf_v)_W)}e_2$ by tameness, and the claim follows.

If $e\not\sim_P e'$, then $e$ and $e'$ lie in distinct components of $(\graf_v)_W$. Given tails $e_1$ and $e_2$ of $\graf$ lying in the same block of $P_{e\sim e'}$, one of the following situations obtains, up to relabeling: either $e_1\sim_P e_2$, or else $e_1\sim_P e$ and $e_2\sim_P e'$. Since all four are tails of $\graf_v$, it follows in either case that $e_1\sim_{\pi_0(\graf_W)}e_2$. The verification that $P_{e\sim e'}<\pi_0(\graf_W)$ is left to the reader.
\end{proof}

This result has an important consequence for modules, which is the motivation for the condition on tails in the definition of tameness.

\begin{lemma}\label{lem:tame quotient}
If $M$ is $i$-tame with respect to $\graf_v$, then $M/(e-e')$ is $i$-tame with respect to $\graf$.
\end{lemma}
\begin{proof}
Since $R[P]/(e-e')\cong R[P_{e \sim e'}]$, Lemma \ref{lem:tame quotient partitions} implies the claim for $M=R[P]$ with $P$ an $i$-tame partition of $E(\graf_v)$.
Therefore, it suffices to show that the collection of modules $M$ for which $M/(e-e')\in \cC(\graf, i)$ forms a Serre subcategory. From an exact sequence \[0\to M\to N\to P\to 0\] of $R[E(\graf_v)]$-modules, there arises the exact sequence \[\mathrm{Tor}(P)\to M/(e-e')\to N/(e-e')\to P/(e-e')\to 0,\] where $\mathrm{Tor}$ indicates the $(e-e')$-torsion submodule. Assuming that $M/(e-e')$ and $P/(e-e')$ lie in $\cC(\graf, i)$, it is immediate that $N/(e-e')$ does as well. On the other hand, if $N/(e-e')\in \cC(\graf, i)$, then $P/(e-e')\in \cC(\graf,i)$; therefore, since $\mathrm{Tor}(P)$ injects into $P/(e-e')$, it follows that $M/(e-e')\in \cC(\graf,i)$, as desired.
\end{proof}

\begin{corollary}\label{cor:tame torsion}
Suppose that $M\in \cC_R(\graf_v, i)$ is $(e-e')$-torsion. With the induced $R[E(\graf)]$-module structure, $M\in \cC_R(\graf,i)$.
\end{corollary}

The main technical input is the following result, whose proof we defer to Section \ref{section:tame tori proof}.

\begin{lemma}\label{lem:tame tori}
The intersection $T_{i-1}(\graf_v)\cap\ker(e-e') $ is $i$-tame with respect to $\graf$.
\end{lemma}

Taking this result for granted, we complete the argument.

\begin{proof}[Proof of Theorem \ref{thm:tame}]
The quotient vanishes when $i=0$, so assume $i>0$. We proceed by induction on $b_1(\graf)$, the base case of a tree being trivial, since the quotient again vanishes. For the induction step, we choose an edge $e$ as above and subdivide with the addition of the bivalent vertex $v$. Since $\graf$ is a connected graph with $b_1(\graf)>0$, we may choose $e$ so that $\graf_v$ is also connected, and an Euler characteristic calculation shows that $b_1(\graf_v)=b_1(\graf)-1$.

Suppose that the claim holds for $\graf_v$ in every degree, and consider the short exact sequence \[0\to \im(\iota)\to H_i(B(\graf))\to \im(\psi)\to 0\] arising from Proposition \ref{prop:vertex explosion}.

We claim first that $\im(\psi)\in \cC_R(\graf,i)$. The induction hypothesis implies that $H_{i-1}(B(\graf_v))/T_{i-1}(\graf_v)\in \cC_R(\graf_v,i-1)$. Since this category is closed under subobjects, the third entry in the exact sequence \[0\to T_{i-1}(\graf_v)\cap \im(\psi)\to  \im(\psi)\to \im(\psi)/T_{i-1}(\graf_v)\cap \im(\psi)\to0\] is $(i-1)$-tame with respect to $\graf_v$, hence $(i-1)$-tame with respect to $\graf$ by Corollary \ref{cor:tame torsion}---we use that $\im(\psi)=\ker(e-e')$ by Proposition \ref{prop:vertex explosion}. Lemmas \ref{lem:tame move up} and \ref{lem:tame tori} now show that the first and third entries are $i$-tame with respect to $\graf$, and $\cC_R(\graf,i)$ is closed under extensions.

It now suffices to show that $\im(\iota)/T_i(\graf)\in \cC_R(\graf,i)$. Write $Q=H_i(B(\graf_v))/T_i(\graf_v)$, and consider the commuting diagram
\[
\begin{tikzcd}
&T_i(\graf_v)/(e-e')\ar{d}\ar{r}&H_i(B(\graf_v))/(e-e')\ar{d}\ar{r}&Q/(e-e')\ar[equal]{d}\ar{r}&0\\
0\ar{r}&T_i(\graf)\ar{r}{\subseteq}&\im(\iota)\ar[dashed]{r}&Q/(e-e')\ar{r}&0,
\end{tikzcd}\] where the vertical arrows are induced by $\iota$. By Proposition \ref{prop:vertex explosion}, the middle arrow is an isomorphism, so the dashed arrow is determined. The top sequence is right exact, since it is obtained by tensoring a short exact sequence down to $R[E(\graf)]$. Since the lefthand vertical arrow is surjective, it follows that the bottom sequence is right exact, hence short exact. Since $Q/(e-e')\in \cC_R(\graf,i)$ by induction and Lemma \ref{lem:tame quotient}, the claim follows.
\end{proof}

\section{Tameness and tori}

\subsection{Proof of Lemma \ref{lem:tame tori}}\label{section:tame tori proof} The main input is the following calculation.

\begin{proposition}\label{prop:tame graded kernel}
For any connected graph $\graf$ and well-separating proper subset $W$ of essential vertices, the kernel of the map $T_W(\graf)\to E^\infty_{|W|,0}$ is $(|W|+1)$-tame with respect to $\graf$. If $R$ is a field, then this kernel is also $(\Delta^{W}_\graf-1)$-small.
\end{proposition}

This result also has the following useful consequence.

\begin{corollary}\label{cor:direct sum}
The kernel and cokernel of the natural map \[\bigoplus_WT_W(\graf)\to T_i(\graf)\] are $(i+1)$-tame and $i$-tame, respectively, where $W$ ranges over well-separating sets of essential vertices of cardinality $i$. If $R$ is a field and $\Delta^i_\graf>1$, then both kernel and cokernel are $(\Delta^{i}_\graf-1)$-small.
\end{corollary}
\begin{proof}
If $\Delta^i_\graf=1$, then $i=1$, and the source of the map in question vanishes. Thus, the claim regarding the kernel is vacuous in this case, and the claim regarding the cokernel follows from Lemma \ref{lem:non-rigid tame}, since every star class in $\graf$ is non-rigid.

Assume that $\Delta^i_\graf>1$, so that a well-separating set $W$ exists. The kernel in question is a submodule of the kernel of the composite \[\bigoplus_WT_W(\graf)\to T_i(\graf)\to \bigoplus_W E^\infty_{|W|,0}(W),\] which is $(i+1)$-tame and $(\Delta^i_\graf-1)$-small by Proposition \ref{prop:tame graded kernel}. The claim regarding the cokernel follows from Lemma \ref{lem:non-rigid tame}, since the cokernel is generated by the images of tori supported at non-well-separating vertex sets, which are in particular non-rigid.
\end{proof}

Before turning to the proof of Proposition \ref{prop:tame graded kernel}, we make a few first reductions and establish notation.

At each vertex $w\in W$, choose two half-edges in distinct components of $\graf_W$. We may make this choice so that the edges associated to each pair of half-edges are not both tails; indeed, if $|W|>1$, this claim follows from the assumption that $\graf$ is connected, while assuming otherwise in the case $|W|=1$ leads to the conclusion that $\graf$ has only one essential vertex, in which case $H_1(B(\graf))=E^\infty_{1,0}$, and the conclusion of Proposition \ref{prop:tame graded kernel} is vacuous. 
Let $A_W\subseteq T_W(\graf)$ denote the set of $W$-tori corresponding to these choices, as exhibited in Construction \ref{construction:good tori}.

\begin{lemma}\label{lem:quotient by good tori}
For any connected graph $\graf$ and well-separating set $W$ of essential vertices, $T_W(\graf)/A_W(\graf)$ is $p$-tame with respect to $\graf$, where \[
p=\begin{cases}
|W|+1&\quad W \text{ proper}\\
|W|&\quad \text{otherwise}.
\end{cases}\] Regardless, if $R$ is a field, then $T_W(\graf)/A_W(\graf)$ is $(\Delta^{W}_\graf-1)$-small.
\end{lemma}

There is reason to find this result surprising. For example, if $\graf$ is a tree and $W$ contains an essential vertex of high valence, then there are many rigid $W$-tori not in $A$, and the naive expectation is that the images of such tori in $T_W(\graf)/A_W(\graf)$ should exhibit the fastest possible growth. As the proof will show, it is the relation of Corollary \ref{cor:combined X} that dampens this growth.

\begin{remark}
Given the smallness estimate of Lemma \ref{lem:quotient by good tori}, it is natural to wonder whether $T_W(\graf)/A_W(\graf)$ is always $|W|$-tame. We do not know the answer to this question.
\end{remark}

\begin{proof}[Proof of Proposition \ref{prop:tame graded kernel}] By Lemma \ref{lem:good tori free module} and Corollary \ref{cor:inject into associated graded}, the composite map \[A_W(\graf)\subseteq  T_W(\graf)\to E^\infty_{|W|,0}\] is injective, so the kernel in question is isomorphic to a submodule of $T_W(\graf)/A_W(\graf)$, and the claim follows from Lemma \ref{lem:quotient by good tori}. 
\end{proof}

Order the half-edges at each $w\in W$ subject to the requirement that, for each $w$, the first two half-edges in the ordering are the two privileged in the construction of $A_W(\graf)$. Given tuples $I=(i_w)_{w\in W}$, $J=(j_w)_{w\in W}$, and $K=(k_w)_{w\in W}$ such that $1\leq i_w,j_w,k_w\leq d(w)$ and $i_w<j_w<k_w$, we have the $W$-torus $\alpha_{IJK}=\bigotimes_{w\in W} \alpha_{i_wj_wk_w}$, and every $W$-torus is uniquely of the form $\pm \alpha_{IJK}$ for some such choice. 

There are two important observations to be made about these indices. First, Lemma \ref{lem:unstable X} implies that every $W$-torus is a linear combination of $W$-tori $\alpha_{IJK}$ satisfying $i_w\equiv 1$. Second, if $i_w\equiv 1$, then $\alpha_{IJK}\in A_W$ if and only if $j_w\equiv 2$. We introduce a filtration \[A_W(\graf)\subseteq M_0\subseteq M_1\subseteq \cdots\subseteq M_{|W|}\subseteq T_W(\graf)\] by declaring $M_r$ to be generated over $R[E]$ by those $W$-tori $\alpha_{IJK}$ such that $i_w\equiv 1$ and $\#\{w: j_w\neq 2\}\leq r$. Rephrasing the observations above, we have $M_{|W|}=T_W(\graf)$ and $M_0=A_W(\graf)$. 
\begin{lemma}\label{lem:tame pieces}
The quotient $M_r/M_{r-1}$ is $p$-tame for every $1\leq r\leq |W|$, where $p$ is as in Lemma~\ref{lem:quotient by good tori}. If $R$ is a field, the quotient is $(\Delta^{W}_\graf-1)$-small.
\end{lemma}
\begin{proof}Since this quotient is generated by the images of those $\alpha_{IJK}$ such that $i_w= 1$  for $w\in W$ and $\#\{w: j_w\neq 2\}= r$, it suffices to show that \[R[E]\cdot [\alpha_{IJK}]\in \cC_R(\graf,p)\] for such $I$, $J$, and $K$, where $[\alpha_{IJK}]=\alpha_{IJK}+M_{r-1}$. Suppose that $j_{w_0}\neq 2$, and define $J'$ and $K'$ by \[j_w'=\begin{cases}
2&\quad w=w_0\\
j_w&\quad w\neq w_0
\end{cases}\qquad \text{ and }\qquad k_w'=\begin{cases}
j_{w_0}&\quad w=w_0\\
k_w&\quad w\neq w_0.
\end{cases}
\] Corollary \ref{cor:combined X} implies that \[(e_2-e_1)\alpha_{IJK}=(e_{j_{w_0}}-e_1)\alpha_{IJ'K}-(e_{k_{w_0}}-e_1)\alpha_{IJ'K'}\in M_{r-1},\] where each edge shown is incident on $w_0$ with associated half-edge given by its subscript. Thus, we have the extra relation $(e_1-e_2)[\alpha_{IJK}]=0$ in the quotient.

By assumption, $e_1$ and $e_2$ lie in distinct blocks of $\pi_0(\graf_W)$, and we obtain a new partition $P<\pi_0(\graf_W)$ by identifying these blocks. In the diagram of surjections among $R[E]$-modules \[\begin{tikzcd}
R[E]\langle \alpha_{IJK}\rangle\ar{d}\ar{r}& R[E]\cdot\alpha_{IJK}\ar{r}& R[E]\cdot[\alpha_{IJK}]\\
R[\pi_0(\graf_W)]\langle \alpha_{IJK}\rangle\ar{d}\ar[dashed]{ur}\\
R[P]\langle \alpha_{IJK}\rangle.\ar[dashed]{uurr}
\end{tikzcd}
\] the inner dashed filler is supplied by Observation \ref{obs:torus quotient}, and the outer dashed filler is supplied by the extra relation. The smallness claim follows, and it now suffices to show that $R[P]$ is $p$-tame.

Without loss of generality, $e_1$ is not a tail, and its other endpoint $u\neq w_0$ is essential. If $u\in W$, then it follows easily that $P$ is $|W|$-tame, hence $p$-tame. If $u\notin W$, then $\pi_0(\graf_W)$ is $(|W|+1)$-tame by Lemma \ref{lem:add a vertex}, so Lemma \ref{lem:inequality tame} implies the claim.
\end{proof}

\begin{proof}[Proof of Lemma \ref{lem:quotient by good tori}]
The module in question is $M_{|W|}/M_0$, which admits the finite filtration $M_1/M_0\subseteq \cdots\subseteq M_{|W|}/M_0$ with $p$-tame and $(\Delta^{W}_\graf-1)$-small associated graded by Lemma \ref{lem:tame pieces}.
\end{proof}

We return now to the setting of Lemma \ref{lem:tame tori}. To wit, $e$ is an edge of $\graf$, which we subdivide by the introduction of a bivalent vertex $v$, resulting in edges $e,e'\in E(\graf_v)$.

\begin{lemma}\label{lem:tame no separation}
Let $W$ be a proper set of essential vertices and $\alpha$ a $W$-torus in $\graf_v$. If $e$ and $e'$ lie in the same component of $(\graf_v)_W$, then $R[E(\graf_v)]\cdot \alpha$ is $(|W|+1)$-tame with respect to $\graf$.
\end{lemma}
\begin{proof} The assumption implies that $\alpha=\psi(\tilde\alpha)$, where $\tilde\alpha$ is the external product of $\alpha$ with the loop formed by $e$, $e'$, and any path joining $e$ and $e'$ in $(\graf_v)_W$ (see Proposition \ref{prop:vertex explosion} for a recollection of the map $\psi$). Therefore, by exactness, we have the isomorphism \[R[E(\graf_v)]\cdot \alpha\cong \frac{R[E(\graf)]\cdot \tilde\alpha}{\im(\iota)\cap R[E(\graf)]\cdot \tilde\alpha}.\] Now, $e$ is incident on the essential vertex $w$, so the Q-relation gives the equation $(e_1-e_0)\tilde\alpha=\alpha'$, where $e_1$ is an edge internal to the loop, $e_0\neq e$ is an edge incident on $w$, and $\alpha'$ is the external product of $\alpha$ and a star class at $w$. Since $\alpha'\in \im(\iota)$, we may apply this observation repeatedly to conclude that the quotient shown above admits a surjection from $R[\pi_0(\graf_W)]$, and the claim follows from Lemma \ref{lem:add a vertex}.
\end{proof}

\begin{proof}[Proof of Lemma \ref{lem:tame tori}]
We claim first that $T_W(\graf_v)\cap \ker(e-e')\in \cC(\graf, i)$, where $W$ is a well-separating subset of $i-1$ essential vertices. If $e$ and $e'$ lie in the same component of $(\graf_v)_W$, the claim follows from Lemma \ref{lem:tame no separation}, so assume otherwise. In this case, the intersection in question is a submodule of the kernel of the map $T_W(\graf_v)\to F^W_{i-1}/F^W_{i-2}$, since the target is $(e-e')$-torsion-free. The claim in this case now follows from Proposition \ref{prop:tame graded kernel}, which asserts that this kernel is $i$-tame with respect to $\graf_v$, hence also with respect to $\graf$ by Lemma \ref{lem:tame quotient}.

Now, summing over well-separating subsets of cardinality $i-1$ and writing $\mathrm{Tor}$ for $(e-e')$-torsion submodules, we have the exact sequences \[\mathrm{Tor}(\ker(f))\to \bigoplus_W\mathrm{Tor}(T_W(\graf_v))\to \mathrm{Tor}\left(\left(\bigoplus_WT_W(\graf_v)\right)/\ker(f)\right)\to \ker(f)/(e-e')\]\[\mathrm{Tor}\left(\left(\bigoplus_WT_W(\graf_v)\right)/\ker(f)\right)\to \mathrm{Tor}(T_{i-1}(\graf_v))\to \mathrm{Tor}(\coker(f)),\] where $f$ is the map of Corollary \ref{cor:direct sum}. It follows from that result that the first and fourth entries of the first sequence are $i$-tame, and we have already shown that the second is so. Therefore, the first entry of the second sequence is $i$-tame. Since the third is so by Corollary \ref{cor:direct sum}, the conclusion follows.
\end{proof}

\subsection{Counting tori} The goal of this section is to calculate the asymptotic dimension of the module $T_i(\graf)$ of tori.

\begin{theorem}\label{thm:counting tori}
Fix a field $\mathbb{F}$ and $i\geq0$. If $\graf$ is a connected graph with an essential vertex and $\Delta^i_\graf>1$, then 
\[\dim T_i(\graf)_k\sim \sum_{W}\frac{1}{(\Delta^i_\graf-1)!}\prod_{w\in W} (d(w)-2)k^{\Delta^i_\graf-1},\] where $W$ ranges over sets of vertices of cardinality $i$ such that $\Delta^W_\graf=\Delta^i_\graf$.
\end{theorem}

\begin{proof}[Proof of Theorem \ref{thm:growth}] By Theorems \ref{thm:polynomial degree} and \ref{thm:counting tori}, the conclusion of the theorem is equivalent to the claim that \[\lim_{k\to\infty}\frac{\dim H_i(B_k(\graf))- \dim T_i(\graf)_k}{k^{\Delta^i_\graf-1}}=0.\] This claim is in turn equivalent, by rank-nullity, to the claim that $H_i(B(\graf))/T_i(\graf)$ is $(\Delta^i_\graf-1)$-small, which follows from Proposition \ref{prop:tame growth} and Theorem \ref{thm:tame}.
\end{proof}

The key ingredient in the proof of Theorem \ref{thm:counting tori} is the following local version, which is interesting in its own right.

\begin{proposition}\label{prop:local count}
For any well-separating set $W$ of essential vertices, \[\dim T_W(\graf)_k\sim \frac{1}{(\Delta^W_\graf-1)!}\prod_{w\in W} (d(w)-2)k^{\Delta^W_\graf-1}.\]
\end{proposition}
\begin{proof}
We have $\dim A_W(\graf)_k\sim \dim T_W(\graf)_k$ by Observation \ref{obs:torus quotient} and Lemma \ref{lem:quotient by good tori}. On the other hand, Lemma \ref{lem:good tori free module} shows that $A_W(\graf)$ is freely generated over $\mathbb{F}[\pi_0(\graf_W)]$ by the set $A_W$. Therefore, \begin{align*}
\dim T_W(\graf)_k&\sim \dim A_W(\graf)_k\\
&\sim\mathbb{F}[\pi_0(\graf_W)]\langle A_W\rangle_{k-2|W|}\\
&=\binom{k-2|W|+\Delta^W_\graf-1}{\Delta^W_\graf-1}|A_W|\\
&=\frac{(k-2|W|+\Delta_\graf^W-1)\cdots (k-2|W|+1)}{(\Delta^W_\graf-1)!}\prod_{w\in W}(d(w)-2)\\
&\sim \frac{1}{(\Delta^W_\graf-1)!}\prod_{w\in W} (d(w)-2)(k-2|W|)^{\Delta^W_\graf-1}\\
&\sim\frac{1}{(\Delta^W_\graf-1)!}\prod_{w\in W} (d(w)-2)k^{\Delta^W_\graf-1},
\end{align*} as claimed.
\end{proof}

\begin{proof}[Proof of Theorem \ref{thm:counting tori}]
By Corollary \ref{cor:direct sum}, $\dim T_i(\graf)_k\sim \sum_W \dim T_W(\graf)_k$, where $W$ ranges over \emph{all} well-separating subsets of cardinality $i$ (here we use the assumption on $\Delta^i_\graf$). By Proposition \ref{prop:local count}, the terms with $\Delta^W_\graf<\Delta^i_\graf$ are negligible in the limit, and the claim follows upon substituting the formula of Proposition \ref{prop:local count}.
\end{proof}

\bibliographystyle{amsalpha}
\bibliography{references}

\end{document}